\documentclass[a4paper,USenglish,numberwithinsect]{lipics}

\usepackage{macros}

\usepackage[draft]{microtype}

\usepackage{hyperref}

\hypersetup{
    colorlinks=true,
    linkcolor=black,
    citecolor=black,
    filecolor=black,
    urlcolor=black,
}

\usepackage[draft]{fixme}

\usepackage{ifthen}
\newboolean{full}
\setboolean{full}{true}

\title{Finitary Corecursion for the \newline Infinitary Lambda Calculus\footnote{This work is supported by the Deutsche Forschungsgemeinschaft (DFG) under project MI~717/5-1}}
\titlerunning{Finitary Corecursion for the Infinitary Lambda Calculus}

\author[1]{Stefan Milius}
\author[1]{Thorsten Wißmann}
\affil[1]{Lehrstuhl für Theoretische Informatik, FAU Erlangen-N\"urnberg, Germany}
\authorrunning{S.~Milius and T.~Wi\ss\/mann}
\Copyright{Stefan Milius and Thorsten Wi\ss\/mann}



\begin{document}
%
%
\FXRegisterAuthor{sm}{asm}{SM}
\FXRegisterAuthor{tw}{atw}{TW}

\maketitle

\begin{abstract}
  Kurz et al.~have recently shown that infinite λ-trees with finitely many free variables modulo α-equivalence form a final coalgebra for a functor on the category of nominal sets. Here we investigate the rational fixpoint of that functor. We prove that it is formed by all rational λ-trees, i.e.~those λ-trees which have only finitely many subtrees (up to isomorphism). This yields a corecursion principle that allows the definition of operations such as substitution on rational λ-trees. 
\end{abstract}

\section{Introduction}

One of the most important concepts in computer science is the $\lambda$-calculus. It is a very simple notion of computation because its syntax consists only of three constructs: variables, $\lambda$-abstraction and function application, and its semantics consists of only two concepts $\alpha$-conversion for renaming of bound variables and $\beta$-conversion for executing function applications. Yet it is very powerful since it is Turing complete and allows to define many notions of higher level programming languages such as booleans, if-then-else, natural numbers, arithmetic operations, lists including mapping and folding, recursion etc.\footnote{Depending on the application a third semantic concept, $\eta$ conversion, may be of interest. But this is neither needed for Turing completeness nor for our work.}

However, whenever one wants to deal with inductive and coinductive definitions in the presence of variable binding subtle issues arise and one has to be careful not to mess up the variable binding. One solution to these problems has been proposed by Gabbay and Pitts~\cite{gp99}. They use \emph{nominal sets} as a framework for dealing with binding operators, abstraction and structural induction. Nominal sets go back to Fraenkel's and Mostowski's permutation model for set theory devised in the 1920s and 1930s.  They are sets equipped with an action of the group of finite permutations on a given fixed set $\V$ of atoms (here these play the role of variables). For an arbitrary nominal set one can then define the notions of ``free'' and ``bound'' variables using the notion of support (we recall this in Section~\ref{sec:nom}). Gabbay and Pitts then consider the functor 
\[
L_\alpha X= \V + [\V]X + X \times X
\]
expressing the type of the term constructors of the \l-calculus (note that the abstraction functor $[\V]X$ is a quotient of $\V \times X$ modulo renaming ``bound'' variables). And they prove that the initial algebra for $L_\alpha$ is formed by all \l-terms modulo $\alpha$-equivalence. 

Recently, Kurz et al.~\cite{nomcoalgdata} have characterized the final coalgebra for $L_\alpha$ (and more generally, for functors arising from so-called binding signatures): it is carried by the set of all infinitary \l-terms (i.e.~finite or infinite \l-trees) with finitely many free variables modulo $\alpha$-equivalence. This then allows to define operations on infinitary \l-terms by coinduction, for example substitution and operations that assign to an infinitary \l-term its normal form computations (e.g.~the B\"ohm, Levy-Longo, and Berarducci trees of a given infinitary \l-term).

Our contribution in this paper is to give a characterization of the \emph{rational fixpoint} of the functor $L_\alpha$. In general, the rational fixpoint for a functor $F$ lies between the initial algebra and the final coalgebra for $F$. If one thinks of it as a coalgebra, it is characterized as the final locally finitely presentable $F$-coalgebra. Intuitively, one may think of it as collecting all behaviours of ``finite'' (more technically, finitely presentable carried) $F$-coalgebras. Examples include regular languages, eventually periodic and rational streams, rational formal power-series etc. For a polynomial endofunctor $F_\Sigma$ on sets associated to the signature $\Sigma$, the rational fixpoint consists of regular $\Sigma$-trees of Elgot~\cite{elgot}, i.e.~those (finite and infinite) $\Sigma$-trees having only finitely many different subtrees (up to isomorphism). We will prove in Section~\ref{sec:rat} that the rational fixpoint for $L_\alpha$ on $\Nom$ is carried by all rational \l-trees modulo $\alpha$-equivalence. Before that we recall in Section~\ref{sec:pre} preliminaries on the infinitary \l-calculus, nominal sets and the rational fixpoint. The finality principle of the rational fixpoint may be understood as a finitary corecursion principle. In Section~\ref{sec:app} we show applications of our main result, in particular, that the coinductive definition of substitution given in~\cite{nomcoalgdata} restricts to rational trees. We also discuss coinductive definitions concerning normal form computations. We conclude in Section~\ref{sec:con}.


{\bf Related work.} The work presented here is based on the second author's student project reported in~\cite{wissmann2014}. 

A related approach to variable binding operations which uses presheaves over finite sets was proposed by Fiore, Plotkin and Turi~\cite{fpt99}. By now this has developed into a respectable body of work by these and other authors. Most related to our work here is the coinductive approach to infinitary and rational \l-terms studied by Ad\'amek, Milius and Velebil~\cite{highrecursion}. This work considers an endofunctor very similar to $L_\alpha$ but on the category of presheafes on finite sets. Its final coalgebra is shown to be the presheaf of all infinite \l-trees and the rational fixpoint the presheaf of all rational trees -- each of them modulo $\alpha$-equivalence. 

\iffull\else Omitted proofs and details may be found in the full version~\cite{mw15_full} of our paper.\fi

\enlargethispage{20pt} 
\section{Preliminaries}
\label{sec:pre}

We assume that readers are familiar with basic notions of category theory and with algebras and coalgebras for an endofunctor. For a given endofunctor $F$ on the category $\C$ we will write $t: \nu F \to F(\nu F)$ for the final coalgebra (assuming that it exists). Given an $F$-coalgebra $(C,c)$ we write $c^\dagger: (C,c) \to (\nu F, t)$ for the unique $F$-coalgebra homomorphism from $C$ to $\nu F$. The category of coalgebras for an endofunctor $F$ is denoted by $\CoAlg\, F$. For introductory texts on coalgebras see~\cite{rutten00,coalgebratutorial,adamek05}.
\takeout{ 
Recall that a \emph{coalgebra} for an endofunctor $F: \C \to \C$ is a pair $(C, c)$ consisting of an object $C$ of $\C$ and a morphism $c: X \to FX$ called the \emph{structure} of the coalgebra. A \emph{homomorphism} of coalgebra from $(C, c)$ to $(D, d)$ is a $\C$-morphism $f: C \to D$ such that $d \cdot f = Ff \cdot c$. An very important concept is that of a \emph{final} coalgebra, i.\,e., an $F$-coalgebra $t: \nu F \to F(\nu F)$ such that for every $F$-coalgebra $(C,c)$ there exists a unique homomorphism $c^\dagger: (C,c) \to (\nu F, t)$. Intuitively, an $F$-coalgebra $(C,c)$ can be thought of as a dynamic system with an object $C$ of states and with observations about the states (e.g.~output, next states etc.) given by $c$. The type of observations that can be made about a dynamic systems is described by the functor $F$. We denote by $\CoAlg\, F$ the category of $F$-coalgebras and their homomorphisms. For more intuition and concrete examples we refer the reader to introductory texts on coalgebras such as~\cite{rutten00,coalgebratutorial,adamek05}.}

We will now give some background on the (infinitary) \l-calculus, on nominal sets and on the rational fixpoint of a functor as needed in the present paper. 

\subsection{Infinitary $\lambda$-Calculus and Rational Trees}

Before we talk about infinitary \l-terms (aka \l-trees) first recall that ordinary \l-terms are defined starting from a fixed countable set of variables $\V$ by the grammar
\[
T ::= x \mid \lambda x.T \mid TT,
\]
where $x$ ranges over $\V$. We denote the set of all \l-terms by $\Lambda$. Free and bound variables and substitution are defined as usual with the operator $\lambda x.(-)$ binding $x$ in its argument. Often one considers \l-terms modulo \emph{$\alpha$-equivalence}, i.e., the least equivalence relation on \l-terms identifying two terms that arise by consistently renaming bound variables. One can think of a term $\lambda x. T$ as representing a computation that takes a parameter $P$ that is used in all free occurences of $x$ in $T$. Hence, the main computation rule of the \l-calculus is $\beta$-reduction, i.e.~the rule
\[
(\lambda x. T)P  \to_\beta T[x \mapsto P].
\]
For example we have $(\lambda x.\lambda y. x)\, a\, b \rightarrow_\beta (\lambda y.a)\, b \rightarrow_\beta a$, where $a$ cannot be reduced further. However, terms may have infinite reduction sequences; a prominent example is $Yf$ for the $Y$-combinator defined as $Y := \lambda g. (\lambda x.
g(x\,x))\,(\lambda x.g(x\,x))$ we have:
\begin{align*}
  Yf & = 
    (\lambda g. (\lambda x. g(x\,x))\,(\lambda x.g(x\,x))) f
    \\
    & \rightarrow_\beta (\lambda x. f(x\,x))\,(\lambda x.f(x\,x))
    \rightarrow_\beta f((\lambda x. f(x\,x))\,(\lambda x.f(x\,x)))
    \\
    &
    \rightarrow_\beta f(f((\lambda x. f(x\,x))\,(\lambda x.f(x\,x))))
    \rightarrow_\beta \cdots
\end{align*}
Informally speaking, this ``converges'' to the infinite term $f (f (f(\cdots)))$. If one takes such infinite terms as legal objects of the \l-calculus one is led to \emph{infinitary} \l-calculus. There one replaces $\l$-terms by (finite and infinite) \l-trees. A \emph{\l-tree} is a rooted and ordered tree with leaves labelled by variables in $\V$ and with two sorts of inner nodes: nodes with one successor labelled by $\lambda x$ for some variable $x \in \V$ and nodes with two successors labelled by $@$. For example, we have the \l-trees
\begin{equation}\label{eq:trees}
\begin{tikzpicture}[lambdatree, level distance =6mm]
            \node (z) {@}
            child { node {\(\l x\)}
                child { node {@}
                    child { node {\(x\)} }
                    child { node {\(x\)} }
                }
            }
            child { node {\(\l x\)}
                child { node {@}
                    child { node {\(x\)} }
                    child { node {\(x\)} }
                }
            }
            ;
          \end{tikzpicture}
          \qquad\qquad
          \begin{tikzpicture}[lambdatree, level distance =6mm]
            \node (z) {@}
            child { node {\(f\)} }
            child { node {@}
                child { node {\(f\)} }
                child { node {@}
                    child { node {\(f\)} }
                    child { node (ddots) {} }
                }
            }
            ;
            \node[anchor=north west] at (ddots.north) {$\ddots$};
        \end{tikzpicture}
\end{equation}
representing the \l-term $(\l x. xx)(\l x. xx)$ and the infinite term $f (f (f(\cdots)))$, respectively. Let $\Lambda^\infty$ be the set of all \l-trees. The notions of free and bound variables of a \l-tree are clear: a variable $x$ is bound in a \l-tree $t$ if there is a path from a leave labelled by $x$ to the root of $t$ that contains a node labelled by $\lambda x$, and $x$ is free in $t$ if there is a path from an $x$-labelled leaf to the root of $t$ that does not contain any node labelled by $\l x$. 

The classic approach to defining operations such a substitution on \l-trees uses that $\Lambda^\infty$ is the metric completion of $\Lambda$ under a natural metric; this idea of using a metric approach to dealing with infinite trees goes at least back to Arnold and Nivat~\cite{ArnoldN80}. Thus, every infinite \l-tree is regarded as the limit of the Cauchy sequence of its truncations at level $n$. Notions such as $\alpha$-equivalence and substitution of \l-trees are then defined by extending the corresponding notions on finite \l-trees (i.e.~\l-terms) continuously. More concretely, two \l-trees $s$ and $t$ are $\alpha$-equivalent iff for every natural number $n$ the pair of truncations at level $n$ of $s$ and $t$ are $\alpha$-equivalent \l-terms (see~\cite[Definition~5.17]{nomcoalgdata}). 

Our aim in this paper is to give a coalgebraic characterization of an important subclass of all \l-trees, the so called \emph{rational} \l-trees. The following definition follows Ginali's characterization~\cite{ginali} of regular $\Sigma$-trees for a signature $\Sigma$:
\begin{definition}
  A \l-tree having only finitely many subtrees (up to isomorphism) is called \emph{rational}. A \l-tree modulo $\alpha$-equivalence, i.e.~an $\alpha$-equivalence class of \l-trees, is called \emph{rational} if it contains at least one rational \l-tree.
\end{definition}
Intuitively, the rational \l-trees are those \l-trees that admit a finite representation as a \l-tree with ``uplinks''. All finite \l-trees are, of course, rational, and so is the right-hand \l-tree in~\eqref{eq:trees}. Other examples are in Figure~\ref{fig:ratexamples}.
\begin{figure}
  \begin{tabular}{lllll}
    \begin{tikzpicture}[lambdatree, level distance=7mm]
      \node (z) {@}
      child { node {\(f\)} }
      child[noedge] { node {} }
      ;
      \begin{scope}
        \path[use as bounding box]
        ($ (z) + (-1cm,5mm) $) rectangle (z);
        \path[draw] (z) .. controls ($ (z) - (-1cm,2cm) $)
        and ($ (z) + (1cm,2cm) $) .. (z);
      \end{scope}
    \end{tikzpicture}
    &
      \begin{tikzpicture}[lambdatree, level distance=7mm]
    \node (z) {@}
    child { node {\(f\)} }
    child {node (mr) {@}
        child { node {\(f\)} }
        child[noedge] { node {} }
    }
    ;
    \begin{scope}
    \path[use as bounding box]
        ($ (z) + (-1cm,5mm) $) rectangle (z);
    \path[draw] (mr) .. controls ($ (mr) - (-1cm,2cm) $)
                             and ($ (z) + (1cm,2cm) $) .. (z);
    \end{scope}
\end{tikzpicture}
    &
\begin{tikzpicture}[lambdatree, level distance=7mm]
    \node (z) {\(\l x\)}
    child { node (at) {@}
        child { node (ml) {\(\l x\)}
            child[noedge] { node (bl) {}}
        }
        child { node (mr) {@}
            child { node {\(x\)}}
            child[noedge] { node (br) {}}
        }
    }
    ;
    \begin{scope}
    \path[use as bounding box]
        ($ (z) + (-1cm,5mm) $) rectangle (br);
    \path[draw] (ml) .. controls ($ (ml) - (1cm,2cm) $)
                             and ($ (z) + (-1cm,1cm) $) .. (at);
    \path[draw] (mr) .. controls ($ (mr) - (-1cm,2cm) $)
                             and ($ (at) + (1cm,1cm) $) .. (at);
    \end{scope}
\end{tikzpicture}
    &
      \begin{tikzpicture}[lambdatree, level distance=7mm]
    \node (z) {@}
    child { node {\(\l x\)}
        child { node {@}
            child { node (bl) {\(\l y\)} }
            child { node {\(y\)} }
        }
    }
    child { node {\(x\)} }
    ;
    \begin{scope}
    \path[use as bounding box]
        ($ (z) + (1cm,5mm) $) rectangle ($ (bl) - (5mm,5mm) $);
    \path[draw] (bl) .. controls ($ (bl) - (1cm,2cm) $)
                             and ($ (z) + (-1cm,1cm) $) .. (z);
    \end{scope}
\end{tikzpicture}
    &
\begin{tikzpicture}[lambdatree, level distance=7mm]
    \node (z) {\(\l x\)}
    child { node (at) {\(\l y\)}
        child { node (mr) {@}
            child[noedge] { node {}}
            child[noedge] { node (br) {}}
        }
    }
    ;
    \begin{scope}
    \path[use as bounding box]
        ($ (z) + (-4mm,5mm) $) rectangle (br);
    \path[draw] (mr) .. controls ($ (mr) - (-5mm,2cm) $)
                             and ($ (z) + (5mm,1cm) $) .. (z);
    \path[draw] (mr) .. controls ($ (mr) - (5mm,2cm) $)
                             and ($ (at) + (-5mm,1cm) $) .. (at);
    \end{scope}
\end{tikzpicture}
\takeout{ 
\\[-2mm]
    \subfloat[ ]{
    \label{fig:floop1}
    }
    &
    \subfloat[ ]{
    \label{fig:floop2}
    }
    &
    \subfloat[ ]{
    \label{fig:ratex1}
    }
    &
    \subfloat[ ]{
    \label{fig:ratex2}
    }
    &
    \subfloat[ ]{
    \label{fig:ratex3}
    }}
\end{tabular}
    \caption{Finite representations of rational \l-trees}
    \label{fig:ratexamples}
\end{figure}
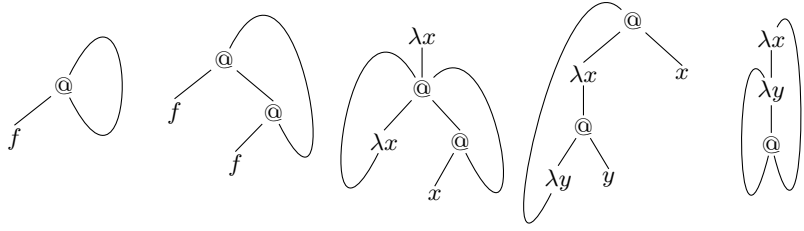
The uplink from some node $s$ to some other node $r$ indicates that the entire tree
starting at $r$ occurs as a subtree of $s$. In other words, such a \l-tree with
uplinks represents its tree unravelling, i.e.~the first and second tree on the left both represent the rational infinite \l-tree shown in~\eqref{eq:trees} on the right.

Things get more complicated, if abstractions come into play, as in the third tree. 
Here the $x$ clearly refers to the $\l x$ in the root, but
some of the \textqt{copies} of $x$ are bound by the $\l x$ in the left branch
and other copies are bound to the abstraction in the root. Something similar can
be observed in the last but one tree, which has two free variables $x,y$, but all
\textqt{copies} of $x$ and $y$ are bound by the previous copy of $\l x$ and $\l
y$ respectively. Finally, the rightmost tree represents a \l-tree that consists of
applications and abstractions only:  
\[
    \l x y. (\l y\l y \ldots) (\l x y. (\l y \l y\ldots) (\l x y. (\l y \l y\ldots) \ldots)).
\]

\subsection{Nominal Sets}
\label{sec:nom}

It was the idea of Gabbay and Pitts~\cite{gp99} to use nominal sets as a category-theoretic framework in which to describe \l-terms modulo $\alpha$-equivalence as the initial algebra for a functor $L_\alpha$. One can then use its universal property to define operations such as substitution of \l-terms. And Kurz et al.~\cite{nomcoalgdata} characterized the final coalgebra for $L_\alpha$; it is carried by the set of \l-trees with finitely many free variables modulo $\alpha$-equivalence. Again, the universal property allows one to define operations such as substitution -- this time by corecursion. We will now recall some background material on nominal sets and the main result of~\cite{nomcoalgdata}.

We fix a countable set \V of variable names. Let $\perms(\V)$ be the group of \emph{finite permutations} of $\V$, where a permutation $\pi \in \perms(\V)$ is called finite iff $\{v\in \V \mid \pi (v) \neq v \}$ is a finite set. 
Now consider a set $X$ together with a group action $\cdot:\perms(\V)\times X \to X$. 
Intuitively, one should think of $X$ as a set of terms, and for a finite permutation of variable names $\pi$ and some term $x$, $\pi\cdot x$ denotes the new term obtained after renaming the variables in $x$ according to $\pi$. In order to talk about variables \textqt{occurring} in $x\in X$ we can check which variable renamings fix the term $x$. This is captured by the notion of \emph{support}: a set $S\subseteq \V$ \emph{supports} $x\in X$ if for all $\pi\in\perms(\V)$ with $\pi(v) = v$ for all $v\in S$ we have $\pi\cdot x = x$.
Some $x\in X$ is \emph{finitely supported} if there is a finite $S\subseteq \V$ supporting $x$.

A \defemph{nominal set} is a set $X$ together with a $\perms(\V)$-action such that all elements of $X$ are finitely supported.

\begin{example}
  \begin{enumerate}
  \item The set \V of variable names with the group action given by $\pi\cdot v =
    \pi (v)$ is a nominal set; for each $v_i\in \V$ the singleton $\{v_i\}$ supports $v_i$.
  \item  Every ordinary set $X$ can be made a nominal set by equipping it with the trivial
    action $\pi\cdot x = x$ for all $x\in X$ and $\pi \in \perms(\V)$. So each $x\in
    X$ can be thought of a term not containing any variable, i.e. the empty set
    supports $x$.
  \item The set $\Lambda$ of all \l-terms forms a nominal set with the group action given by renaming of free variables. Every \l-term is supported by the set of its free variables. In contrast the set $\Lambda^\infty$ of all \l-trees is not nominal since \l-trees with infinitely many free variables do not have finite support. However, the set $\Lambda^\infty_\mathsf{ffv}$ of all \l-trees with finitely many free variables is nominal.
  \end{enumerate}
\end{example}

Notice that if $S\subseteq V$ supports $x\in X$, then $S'\supseteq S$ also supports $x\in X$. So $S$ supporting $x$ only means that by not touching the members of $S$ one does not modify the term $x$. But it is more interesting to talk about the variables actually occurring in $x$. This is achieved by considering the smallest set supporting $x$, which is denoted by $\supp(x)$. If $v\in \V\setminus \supp(x)$, we say that $v$ is \emph{fresh for} $x$, denoted by $v\fr x$.

\begin{example}
    The set $\Potf(\V)$ of finite subsets of $\V$, together with the
    point-wise action is a nominal set. The support of each $u \in \Potf(\V)$ is $u$
    itself:
    $\pi \cdot u = \{ \pi \cdot x\mid x\in u\}$ and $\supp(u) = u$.
    Note that $\Pot(\V)$ with the point-wise action is not a nominal set because the
    infinite $\{v_0,v_2,v_4,\ldots\}$ does not have any finite support.
\end{example}

The morphisms of nominal sets are those maps which are equivariant: an \defemph{equivariant map} $f: (X,\cdot) \to (Y,\star)$ is a map $f: X\to Y$ with 
$f(\pi \cdot x) = \pi\star f(x)$ for all $\pi\in \SV, x\in X$.

For example, the function $\supp: X \to \Potf(\V)$ mapping each element to its (finite) support is an equivariant map.

\begin{rem}
    For any equivariant $f: (X,\cdot) \to (Y,\star)$, we have $\supp(f(x))
    \subseteq \supp(x)$ for any $x\in X$.
\end{rem}
The nominal sets -- together with the equivariants as morphisms -- form a
category, denoted by \Nom. As shown in~\cite{gmm06}, this category is (equivalent to) a Grothendieck topos (the so-called Shanuel topos), and so it has rich categorical structure. We only mention some facts needed for the current paper. 

Monomorphisms and epimorphisms in $\Nom$ are precisely the injective and surjective equivariant maps, respectively. It is not difficult to see that every epimorphism in $\Nom$ is strong, i.e., it has the unique diagonalization property w.r.t.~any monomorphism: given an epimorphism $e: A \epito B$, a monomorphism $m: C \hra D$ and $f: A \to C$, $g: B \to D$ with $g \cdot e = m \cdot f$, there exists a unique diagonal $d: B \to C$ with $d \cdot e = f$ and $m \cdot d = g$. 

Furthermore, $\Nom$ has image-factorizations; that means that every equivariant map $f: A \to C$ factorizes as $f = m \cdot e$ for an epimorphism $e: A \epito B$ and a monomorphism $m: B \hra C$.
\takeout{ 
    \[
        \begin{tikzcd}[compactcd]
            A \arrow[->>]{r}{e}
        \arrow[shiftarr={yshift=3.5ex} ]{rr}{f}
            &
            B \arrow[hook]{r}{m} &
            C
        \end{tikzcd}
    \]} 
Note that the intermediate object $B$ is (isomorphic to) the image $f[A]$ in $B$ with the restricted action. For an endofunctor $F$ on $\Nom$ preserving monos this factorization systems lifts to $\CoAlg\, F$: every $F$-coalgebra homomorphism $f$ has a factorization $f = m \cdot e$ where $e$ and $m$ are $F$-coalgebra homomorphisms that are epimorphic and monomorphic in $\Nom$, respectively.

\enlargethispage{3mm}
Recall from~\cite[Section 2.2]{pittsbook} that $\Nom$ is complete and cocomplete with colimits and finite limits formed as in $\Set$. In fact, $\Nom$ is a locally finitely presentable category in the sense of Gabriel and Ulmer~\cite{gu71} (see also Ad\'amek and Rosick\'y~\cite{ar94}). We shall not recall that notion here as it is not needed in the current paper; intuitively, a locally finitely presentable category is a category with a well behaved ``finite'' objects (called \emph{finitely presentable} objects) such that every object can be build (as a filtered colimit) from these. Petrişan~\cite[Proposition~2.3.7]{petrisanphd} has shown that the finitely presentable objects of $\Nom$ are precisely the orbit-finite nominal sets.
\begin{definition}
  For a nominal set $(X, \cdot)$ and $x \in X$ the set $\{ \pi \cdot x \mid\pi\in \perms(\V)\}$ is called the \emph{orbit} of $x$. 
  A nominal set $(X,\o)$ is said to be \defemph{orbit-finite} if it has only finitely many orbits.
\end{definition}
The notion of orbit-finiteness plays a central role in our paper since the rational \l-trees modulo $\alpha$-equivalence are described by precisely all the coalgebras with an orbit-finite carrier for the functor $L_\alpha$ further below (cf.~Proposition~\ref{prop:image} and Theorem~\ref{thm:rat}).  

We now collect a few easy properties of orbit-finite sets that we are going to need. \iffull(The proofs can be found in the appendix.)\fi
\begin{lemma}\label{orbitsamesupp}
    For any $x_1,x_2 \in X$ in the same orbit, we have $|\supp(x_1)| = |\supp(x_2)|$.
\end{lemma}
\begin{lemma}\label{sameorbitfaculty}
    For an element $x$ of a nominal set $X$, there are at most $|\supp(x)|!$
    many elements with support $\supp(x)$ in the orbit of $x$.
\end{lemma}
\begin{lemma}\label{contained}
    For a finite set $W\subseteq \V$ and an orbit $\mathcal{O}$ of the nominal
    set $X$ there are only finitely many elements in $\mathcal{O}$ whose
    support is contained in $W$.
\end{lemma}
\takeout{ 
\begin{proposition}\label{xfxsamesupport}
    If $f:X\to X$ is an equivariant and if $x$ and $f(x)$ are in the same orbit,
    then  $\supp(x) = \supp(f(x))$.
\end{proposition}
Note that even though $f:X\to X$ is an equivariant and $x, f(x)$ are in the same
orbit, $x\neq f(x)$ might still hold. For example consider the equivariant
$\op{swap}:\V\times \V\to \V\times \V, (a,b) \mapsto (b,a)$. Then for $a\neq b$:
\[
    \supp(a,b) = \{a,b\} = \supp(b,a)
    \text{ but }
    (a,b) \neq (b,a).
\]
The generalization of that example directly leads to the following.
\begin{proposition}\label{nomblowup}
    For a fixed $x\in X$ there are finitely many equivariant functions
    $f: X \to X$ such that $x$ and $f(x)$ are in the same orbit. In fact,
    there are at most $|\supp(x)|!$ such functions, corresponding to the
    permutations on the support of $x$.
\end{proposition}
\begin{proof}
    For $x\in X$, define the subset of the orbit of $x$ which is
    \textqt{reachable via equivariants} as
    \[
        N = \{ \pi \cdot x \mid \text{ there is some equivariant }
                    f: X\to X\text{ and some }\pi\text{ with }f(x) = \pi\cdot x\}.
    \]
    By \autoref{xfxsamesupport}, $\supp(f(x)) = \supp(x)$. By
    \autoref{sameorbitfaculty}, there are at most $|\supp(x)|!$ elements
    with the same support as $x$ in the same orbit. Hence,
    $N$ is bound by $|\supp(x)|!$.
\end{proof}
} 
Let us now recall from Kurz et al.~\cite{nomcoalgdata} how all \l-trees form a final coalgebra in $\Nom$. First consider the following endofunctor on $\Nom$:
\[
LX = \V + \V \times X + X \times X;
\]
its coproduct components describe the type of the term constructors of the \l-calculus (variables, \l-abstraction and application, respectively). As shown in~\cite{nomcoalgdata}, the final coalgebra for this functor is carried by the set of all \l-trees containing finitely many (free and bound) variables.\footnote{Note that this is different from the set $\Lambda^\infty_\mathsf{ffv}$ mentioned in Example~\ref{ex:coalgs}.3; \l-trees in the latter may have infinitely many bound variables.} Its coalgebra structure is the obvious map decomposing a \l-tree at the root: a single node \l-tree is mapped to its node label in $\V$, a \l-tree whose root is labeled by $\lambda x$ to $(x, t)$, where $t$ is the \l-tree defined by the successor of the root and a \l-tree with root label $@$ to the pair of \l-trees defined by the successors of the root. 

Since this final coalgebra completely disregards $\alpha$-equivalence it is not possible to define substitution as a total operation on it. The solution is to replace the second component $\V \times X$ of $L$ by Gabbay and Pitts abstraction functor~\cite[Lemma~5.1]{gp99} that takes $\alpha$-equivalence into account:
\begin{definition}
  \label{def:abstr}
  Let $(X,\cdot)$ be a nominal set. We define $\alpha$-equivalence $\sim_\alpha$ as the relation on $\V×X$ as
    \[
        (v_1, x_1) \sim_\alpha (v_2, x_2)
        \    
        \text{if there exists } z\fr \{v_1,v_2\}, z\fr x_1, z\fr x_2\text{ with }
            (v_1\ z) x_1 = (v_2\ z) x_2.
    \]
    The $\sim_\alpha$-equivalence class of $(v,x)$ is denoted by $\abstr{v}x$.
    The abstraction $[\V]X$ of the nominal $X$ is the quotient $(\V ×
    X)/\mathord{\sim_\alpha}$ with the group action defined by
    \[
        \pi\cdot \abstr{v}x = \abstr{\pi(v)}(\pi\cdot x).
    \]
    For an equivariant map $f: X\to Y$, $[\V]f: [\V]X \to [\V]Y$ is defined by
    $\abstr{v}x \mapsto \abstr{v}(f(x))$.
\end{definition}
Note that the abstraction functor $[\V](-)$ is \emph{strong}, i.e., we
have a natural transformation $\tau$ with components $\tau_{X,Y}:
[\V]X \times Y \to [\V](X \times Y)$ given by
$\tau_{X,Y}(\abstr{v}x,y) = \abstr{v}(x,y)$; we will need the strength $\tau$ in Section~\ref{sec:sub}.
Now one considers the endofunctor $L_\alpha$ on $\Nom$ given by 
\[
L_\alpha X = \V + [\V] X + X \times X.
\]
Gabbay and Pitts~\cite{gp99} showed that its initial algebra consists of all \l-terms modulo $\alpha$-equivalence, and the main result of Kurz et al.~\cite{nomcoalgdata} is that the final coalgebra $\nu L_\alpha$ is carried by the set $\Lambda^\infty_\mathsf{ffv}$ of all \l-trees with finitely many free variables quotiented by $\alpha$-equivalence. The coalgebra structure is the same as on the final coalgebra for $L$ -- one can show that this is well-defined on equivalence classes modulo $\alpha$-equivalence. 

\subsection{The Rational Fixpoint}

Recall that by Lambek's Lemma~\cite{lambek}, the structure of an initial algebra and a final coalgebra for a functor $F$ are isomorphisms, so both yield \emph{fixpoints} of $F$. Here we shall be interested in a third fixpoint that lies in between initial algebra and final coalgebra called the \emph{rational fixpoint} of $F$. This can on the one hand be characterized as the initial iterative algebra for $F$ (see~\cite{iterativealgebras}) or as the final locally finitely presentable coalgebra for $F$~(see~\cite{m_linexp}). We will only recall the latter description since the former will not be needed in this paper. 

The rational fixpoint can be defined for any \emph{finitary} endofunctor $F$ on a locally finitely presentable (lfp, for short) category $\C$, i.e.~$F$ is an endofunctor on $\C$ that preserves filtered colimits. Examples of lfp categories are $\Set$, the categories of posets and of graphs, every finitary variety of algebras (such as groups, rings, and vector spaces) and every Grothendieck topos (such as $\Nom$). The finitely presentable objects in these categories are: all finite sets, posets or graphs, those algebras presented by finitely many generators and relations, and, as we mentioned before, the orbit-finite nominal sets. 

Now let $F: \C \to \C$ be finitary on the locally finitely presentable category $\C$ and consider the full subcategory $\CoAlgf\, F$ of $\CoAlg\, F$ given by all $F$-coalgebras with a finitely presentable carrier. In~\cite{m_linexp} the \emph{locally finitely presentable} $F$-coalgebras were characterized as precisely those coalgebras that arise as a colimit of a filtered diagram of coalgebras from $\CoAlgf\, F$. It follows that the \emph{final} locally finitely presentable coalgebra can be constructed as the colimit of \emph{all} coalgebras from $\CoAlgf\, F$. More precisely, one defines a coalgebra $r: \varrho F \to F(\varrho F)$ as the colimit of the inclusion functor of $\CoAlgf\, F$:
$
(\varrho F, r) := \colim (\CoAlgf\, F \hra \CoAlg\, F).
$ 
Note that since the forgetful functor $\CoAlg\, F \to \C$ creates all colimits this colimit is actually formed on the level of $\C$. 
The colimit $\varrho F$ then carries a uniquely determined coalgebra structure $r$ making it the colimit above.

As shown in~\cite{iterativealgebras}, $\varrho F$ is a fixpoint for $F$, i.e.~its coalgebra structure $r$ is an isomorphism. From~\cite{m_linexp} we obtain that local finite presentability of a coalgebra $(C,c)$ has the following concrete characterizations: (1)~for $\C = \Set$ local finiteness, i.e.~every element of $C$ is contained in a finite subcoalgebra of $C$; (2)~for $\C = \Nom$, local orbit-finiteness, i.e.~every element of $C$ is contained in an orbit-finite subcoalgebra of $C$; (3)~for $\C$ the category of vector spaces over a field $K$, local finite dimensionality, i.e., every element of $C$ is contained in a subcoalgebra of $C$ carried by a finite dimensional subspace of $C$.
\begin{example}
  \label{ex:coalgs}
  We list only a few examples of rational fixpoints; for more see~\cite{iterativealgebras,m_linexp,bms13}. 
  \begin{enumerate}
  \item Consider the functor $FX = 2 \times X^A$ on $\Set$ where $A$ is an input alphabet and $2 = \{0,1\}$. The $F$-coalgebras are precisely the deterministic automata over $A$ (without initial states). The final coalgebra is carried by the set $\Pow(A^*)$ of all formal languages and the rational fixpoint is its subcoalgebra of regular languages over $A$.
  \item For $FX = \Real \times X$ on $\Set$ the final coalgebra is carried by the set $\Real^\omega$ of all real streams and the rational fixpoint is its subcoalgebra of all eventually periodic streams, i.e.~streams $uvvv\cdots$ with $u, v \in \Real^*$. Taking the same functor on the category of real vector spaces we get the same final coalgebra $\Real^\omega$ with the componentwise vector space structure, but this time the rational fixpoint is formed by all rational streams (see~\cite{rutten_rat,m_linexp}).
  \item Let $\Sigma$ be a signature of operation symbols with prescribed arity, i.e.~a sequence $(\Sigma_n)_{n < \omega}$ of sets. This give rise to an associated polynomial endofunctor $F_\Sigma$ on $\Set$ given by $F_\Sigma X = \coprod_{n < \omega} \Sigma_n \times X^n$. Its initial algebra is formed by all $\Sigma$-terms and its final coalgebra by all (finite and infinite) $\Sigma$-trees, i.e.~rooted and ordered trees such that every node with $n$ children is labelled by an $n$-ary operation symbol. And the rational fixpoint consists precisely of all regular $\Sigma$-trees of Elgot~\cite{elgot} (see also~Courcelle~\cite{courcelle}), i.e.~those $\Sigma$-trees having only finitely many different subtrees up to isomorphism (see Ginali~\cite{ginali}).
  \end{enumerate}
\end{example}

Note that in all the above examples the rational fixpoint $\varrho F$ allways occurs as a subcoalgebra of the final coalgebra $\nu F$. But this need not be the case in general (see~\cite[Example~3.15]{bms13} for a counterexample). However, we have the following result:
\begin{proposition}[{\cite[Proposition 3.12]{bms13}}]
  \label{prop:image}
  Suppose that in $\C$ finitely presentable objects are closed under strong quotients and that $F$ is finitary and preserves monomorphisms. Then the rational fixpoint $\varrho F$ is the subcoalgebra of $\nu F$ given by the union of images of all coalgebra homomorphisms $c^\dagger: (C,c) \to (\nu F, t)$ where $(C,c)$ ranges over $\CoAlgf\, F$.\footnote{In a general lfp category the \emph{image} of $c^\dagger$ is obtained by taking a strong epi-mono factorization of $c^\dagger$, and the \emph{union} is then  obtained as a directed colimit of the resulting subobjects of $(\nu F, t)$.}
\end{proposition}
In particular, for a finitary functor $F$ on $\Set$ (or $\Nom$ resp.) preserving monomorphisms the rational fixpoint is the union in $\nu F$ of images of all finite (or orbit-finite resp.) coalgebras; in symbols: 
\[
\varrho F = \hspace{-10pt}\bigcup\limits_{\text{$(C,c)$ in $\CoAlgf\, F$}} \hspace{-10pt} c^\dagger[C] \ \subseteq\ \nu F.
\]
Note that it is sufficient to let $(C,c)$ range over those coalgebras
in $\CoAlgf\, F$ where $c^\dagger$ is injective (or an inclusion map)
because for an arbitrary (orbit-)finite $(C,c)$ in $\CoAlgf F$ its image
$c^\dagger[C]$ is an (orbit-)finite $F$-coalgebra, too. 
\section{The Rational Fixpoint in Nominal Sets}
\label{sec:rat}

In this section we are going to prove the main result of our paper, a characterization of the rational fixpoint for the functor $L_\alpha$ as the rational \l-trees modulo $\alpha$-equivalence.

But we start with the rational fixpoint of the functor $LX = \V + \V \times X + X \times X$. Note that both functors $L$ and $L_\alpha$ are finitary and preserve monomorphisms (to see this use~\cite[Proposition~5.6]{nomcoalgdata} and the fact the forgetful functor from $\Nom$ to $\Set$ creates colimits).
\begin{proposition}\label{prop:ratL}
  The rational fixpoint of the functor $L: \Nom \to \Nom$ is formed by all rational \l-trees.
\end{proposition}
\takeout{ 
When taking $\alpha$-equivalence into account, we need to use the functor
\[
    L_\alpha X = \V + [\V]X + X×X.
\]
$L_\alpha$ is closely related to $L$ by the natural transformation $q: L \to L_\alpha$
\[
    q_X \equiv
    \big(
    L X =
    \V + \V×X + X×X
    \xrightarrow{\V + [-]_\alpha + X×X}
    \V + [\V]X + X×X
    = L_\alpha \big).
\]} 

In the proof of the following theorem we will slightly abuse notation and consider $LX = \V + \V \times X + X \times X$ as an endofunctor on $\Set$. Note that its final coalgebra is formed by the set $\Lambda^\infty$ of all \l-trees and its rational fixpoint by all rational \l-trees (this follows from Example~\ref{ex:coalgs}.3). 
\begin{theorem}
    Let $X\xrightarrow{a} L_\alpha X$ be an orbit-finite coalgebra. Then for all $\op{root}\in X$,
    $a^\dagger(\op{root}) \in \varrho L_\alpha$ is a rational $\lambda$-tree.
    \label{thm:onlyrational}
\end{theorem}
\begin{proof}[Proofsketch]\iffull(for a full proof see the appendix)\fi
    Let $m := \max_{x\in X}\big|\supp(x)\big|$ be the maximal number of free variables in any element of $X$. This exists by \autoref{orbitsamesupp} since $X$ is orbit-finite.
    Let $W \subseteq \V$ be some set of $m+1$ variables containing $\supp(\op{root})$. Hence for all $x \in X$ there exists a $w \in W$ with $w \fr  x$.

    In the following, one constructs a rational $\l$-tree in the $\alpha$-equivalence class of $a^\dagger(\op{root})$. First, define an $L$-coalgebra $C \xrightarrow{c} LC = \V + \V\times C + C\times C$ in \Set with 
    $ C := \{x \in X\mid \supp(x) \subseteq W \}$ and 
    \[
    c(x) = \begin{cases}
      w & \text{if }a(x) = w \in W \subseteq \V \\
      (\ell, r) &\text{if }a(x) = (\ell,r)\in X\times X\\
      (w,y) &\text{if }a(x) = \abstr{v} y'\text{ and }y = (v\ w)y'\text{ for some }w\in W\setminus\supp(x)
    \end{cases}
    \]
    Next one readily verifies that $c$ is well-defined, i.e., its image really lies in $LC$.

    Furthermore, $C$ is finite because $X$ is orbit-finite and within any orbit there are only finitely many elements with a support contained in $W$ by \autoref{contained}.  Let $c^\dagger$
    denote the unique $L$-coalgebra homomorphism into the final $L$-coalgebra in \Set. Since $C$ is finite, we know that $c^\dagger: C \to \nu L$ factors through the rational fixpoint, i.e.~for every
    $x \in C$, $c^\dagger(x)$ is a rational $\lambda$-tree. One then proves that $\eqa{c^\dagger(x)} = a^\dagger(x)$ for all $x\in C\subseteq X$,  where $\eqa{-}$ denotes $\alpha$-equivalence classes. This involves a non-trivial induction argument using the \emph{final chains} of the set functor $L$ and the functor $L_\alpha$ on $\Nom$ as well as technical details from~\cite{nomcoalgdata}\iffull (see the appendix)\fi. It follows that $a^\dagger(\op{root})$ is rational.
\end{proof}

For the $L$-coalgebra $(C,c)$ in \Set from the previous proof, we know that for
any $x\in C$, the resulting tree $c^\dagger(x)$ has at most $|C|$ subtrees. This
does not hold for an $L_\alpha$-coalgebra $(X,a)$ in \Nom: if $X$ has a
non-trivial action, then the cardinality of $X$ is at least infinite, i.e.~the
cardinality does not give a reasonable bound for the number of subtrees.
And the number of orbits $n$ is not a bound either. The problem is that
multiple elements from the same orbit may represent different subtrees. For
example, consider the rational tree
\[
t :=
\begin{tikzpicture}[lambdatree, level distance=7mm, scale=0.7]
    \node (z) {@}
    child { node {\(v_0\)} }
    child { node {\(v_1\)} }
    ;
\end{tikzpicture}
,
\]
and let $X := \V + \{ (\ell,r) \in \V×\V \mid \ell \neq r\}$ be equipped with the coalgebra structure
\[
    X \xrightarrow{a} L_\alpha X,
    \quad
    a(x) =
    \begin{cases}
        v & \text{ if } x = v\in \V \\
        (\ell,r) & \text{ if }x = (\ell,r) \in \V×\V.
    \end{cases}
\]
$X$ is constructed to have two orbits: one consisting of single variables and
one consisting of unequal ordered pairs of variables. We have $(v_0,v_1) \in X$
and $a^\dagger\big((v_0, v_1)\big) = [t]_\alpha$. But $t$ has three subtrees,
namely $t$ itself, $v_0$, and $v_1$. This example is expanded later in
Example~\ref{bigexample}.

But when looking closer at the construction of $C$ in the previous proof, we can give a bound on
the number of its elements, i.e.~the number of (up to isomorphism) different subtrees of the rational tree $c^\dagger(\op{root})$. 
\begin{proposition}
  \label{prop:bound}
    Let $(X,a)$ be an orbit-finite $L_\alpha$-coalgebra with $n$ orbits and let $m = \max_{x\in X}\big|\supp(x)\big|$. Then the number of elements of the coalgebra $C$ (as constructed in the previous proof) is bounded by $n\cdot (m+1)!$.
\end{proposition}
\begin{proof}
  Recall from the proof of \autoref{thm:onlyrational} that
  $C := \{x \in X\mid \supp(x) \subseteq W \}$, where $W$ is a set of
  $m+1$ variables.  Consider a fixed orbit $O$ whose elements have a
  support of cardinality $k$:
  there are at most $\begin{pmatrix} m+1\\
    k\end{pmatrix}$
  possibilities of choosing a $k$-element subset $S$ of $W$ and for
  any such $S$ there are at most $k!$ elements in $O$ with support
  $S$, by \autoref{sameorbitfaculty}. Thus, the number of elements of $O$ in $C$ is at most
    \[
        {m+1 \choose k} \cdot k!
        = \frac{(m+1)!}{k!\cdot (m+1-k)!}\cdot k!
        = \frac{(m+1)!}{(m+1-k)!}
        \overset{k\le m}{\le}
        \frac{(m+1)!}{(m+1-m)!}
        = (m+1)!.
    \]
    In total, the cardinality of $C$ is bounded by $n\cdot (m+1)!$.
\end{proof}

That the number of orbits $n$ occurs linearly is not surprising,
because if we have some finite carrier set $X$ with the trivial action
that \textqt{uses} all its elements for the coalgebra structure, we
have exactly one subtree per element of $X$. In
Example~\ref{bigexample} we shall see that the factor is $(m+1)!$ is
necessary. But before that we state and prove our main result:
\begin{theorem}
  \label{thm:rat}
    The rational fixed point $\varrho L_\alpha$ contains precisely the rational
    \l-trees modulo $\alpha$-equivalence.
\end{theorem}
\begin{proof}
    After \autoref{thm:onlyrational} it only remains to show that all
    $\alpha$-equivalence classes of rational
    $\l$-trees are in $\varrho L_\alpha$. Let $u_\alpha \in \nu L_\alpha$ be
    rational, witnessed by some rational representative $u \in u_\alpha$ with
    only finitely many subtrees (up to isomorphism). Let $C$ be the finite
    set of all subtrees of $u$ and define the nominal set $X$ as 
    \[
        X := \bigcup_{s\in C} O([s]_\alpha) \ \subseteq\ \nu L_\alpha,
    \]
    where $[s]_\alpha \in \nu L_\alpha$ is the $\alpha$-equivalence class of the
    subtree $s$ and $O(y) \subseteq \nu L_\alpha$ denotes the orbit of a given
    element $y\in \nu L_\alpha$. Note that the group action of $\nu
    L_\alpha$ restricts to $X$ since it is a union of orbits. 

    Next we define a coalgebra structure $a: X \to L_\alpha X$ by
    restriction of the structure $t: \nu L_\alpha \to L_\alpha(\nu
    L_\alpha)$: set $a(x) := t(x)$ for all $x \in X$. We need to check
    that this is well-defined, i.e.~that $t(x)$ really lies in
    $L_\alpha X$ for every $x \in X$. For this we consider three cases:
    \begin{enumerate}
    \item The case $t(x) \in \V$ is clear; 
    \item Suppose that $t(x) = \abstr{v}y \in [\V](\nu L_\alpha)$ where $x = \pi\cdot
    [\l w.s]_\alpha$ for some $\pi \in \perms(\V)$ and some subtree $\l w.s$ of $u$.
    Then we have
    \[
        \abstr{w}[s]_\alpha
        = t([\l w.s]_\alpha)
        = t(\pi^{-1}\cdot x)
        = \pi^{-1} \cdot t(x)
        = \pi^{-1} \cdot \abstr{v}y
        = \abstr{\pi^{-1}(v)}(\pi^{-1} \cdot y).
    \]
    By the definition of abstraction, we have some $z\in \V$ with
    \[
        (w\ z)\cdot [s]_\alpha = (\pi^{-1}(v)\ z)\cdot \pi^{-1} \cdot y.
    \]
    Hence, $y$ is in the orbit of $[s]_\alpha$ and therefore $t(x)$ is in $[\V]X$.

    \item For $t(x) = (\ell,r) \in \nu L_\alpha × \nu L_\alpha$, let
    $x\in X$ be $\pi\cdot [(s_\ell, s_r)]_\alpha$ for some $\pi \in \perms(\V)$ and subtrees $s_\ell, s_r$ of $u$.
    Analogously to the previous case, we have
    \[
        (\ell,r)
        = t(x)
        = t(\pi\cdot [(s_\ell , s_r)]_\alpha)
        = \pi\cdot ([s_\ell]_\alpha, [s_r]_\alpha)
        = (\pi\cdot [s_\ell]_\alpha,\pi\cdot  [s_r]_\alpha)
        \in X×X.
    \]
    \end{enumerate}
    By construction, $u_\alpha \in X$ and $a^\dagger(x) = x$ holds for all $x\in X$. 
    By the finiteness of $C$, $X$ is orbit-finite and thus $a^\dagger[X]
    \subseteq \varrho L_\alpha$, and in particular $u_\alpha \in \varrho L_\alpha$.
\end{proof}

The following example shows that the bound in \autoref{prop:bound} on the number of elements of the $L$-coalgebra $C$ from the proof of \autoref{thm:onlyrational} can essentially not be improved even if we omit the usage of \l-abstraction. 
\begin{example}
    \label{bigexample}
Let $\ell \geq 1$ be a fixed natural number and let $m = 2^{\ell -  1}$. Further let $V = \{v_1, \ldots, v_m \}$ be a set of $m$ variables and consider the rational \l-trees parametrized by permutations $\sigma \in \perms(\V)$ shown in Figure~\ref{fig:bigex}.
\begin{figure}
  \begin{tikzpicture}[ lambdatree,
                        inode/.style={
                            draw=black,
                            fill=black,
                            shape=circle,
                            inner sep=0.5mm,
                        },
                        every edge/.style={
                             draw,
                        },
                      ]
   \node (root) {@$\mathrlap{\ = r_\sigma}$};
   \node (nextroot) [below left of=root,
                     xshift=-5cm,
                     yshift=0cm,
                     anchor=center
                    ] {$\mathllap{h_\sigma =\ }@$};
   \begin{scope}[every node/.style={
                    subtree,
                    minimum width=23mm,
                    minimum height=20mm,
                    inner sep=0cm,
                    yshift=0mm,
                    },
                    node distance=17mm,
                ]
   \node (root12) [below left of=nextroot,anchor=north]
                            {$r_{\sigma\,(1\ 2)}$};
   \node (root1n) [below right of=nextroot,anchor=north]
                    {$r_{\sigma\,(1\ \cdots\ m)}$};
   \end{scope}
   \node (l11) [below right of=root] {@};
   \node (l21) [below left  of=l11,xshift=-5mm] {@};
   \node (l22) [below right of=l11,xshift=5mm] {@};
   \node (l31) [below left  of=l21] {@};
   \node (l32) [below right of=l21] {@};
   \node (l33) [below left  of=l22] {@};
   \node (l34) [below right of=l22] {@};
   \coordinate (rm) at ($ (l32) !.5! (l33) $);
   \coordinate (r1) at ($ (l31) - (5mm,1cm) $);
   \coordinate (r2) at ($ (l34) - (-5mm,1cm) $);
   \coordinate (r1p) at ($ (l31) - (-0mm,1cm) $);
   \coordinate (r2p) at ($ (l34) - (0mm,1cm) $);
   \node (v1) [below left of=r1,anchor=west,yshift=-3mm,xshift=-1mm] {$\sigma v_{1}$};
   \node (v2) [below right of=r1,anchor=east,yshift=-3mm,xshift=1mm] {$\sigma v_{2}$};
   \node (vm) [below right of=r2,anchor=east,yshift=-3mm,xshift=1mm] {$\sigma v_{m}$};
   \node (vdots) at ($ (v1) !.5! (vm) $) {$\cdots$};

   \path[draw,
        ]
        (root) edge (l11)
        (l11) edge (l21)
        (l11) edge (l22)
        (l21) edge (l31)
        (l21) edge (l32)
        (l22) edge (l33)
        (l22) edge (l34)
        (root) edge (nextroot)
        (nextroot) edge (root12.north)
        (nextroot) edge (root1n.north)
        \foreach \x in {l31,l32,l33,l34} {
            (\x) edge ($ (\x) + (-4mm,-1cm) $)
            (\x) edge ($ (\x) +  (4mm,-1cm) $)
        }
        ;
    \path[draw, shorten >=2mm]
        (v1) edge (r1)
        (v2) edge (r1)
        (vm) edge (r2)
    ;

    \node[draw=none,inner sep=0pt,outer sep=0pt,fit=(l11) (v1) (vm)] (boundary) {};
    \draw [decorate,decoration={brace,amplitude=10pt,mirror,raise=4pt},yshift=0pt]
    (boundary.south east) -- (boundary.north east)
        node [black,midway,anchor=west,xshift=17pt,anchor=west,
              text width=2.6cm,align=left] {\footnotesize $\ell$ levels};

    
    \begin{scope}[every node/.style={
                    draw=black,
                    dotted,
                    inner sep=1mm,
                    outer sep=1mm,
                    rounded corners=1mm,
                  },
                  intree/.style={
                    minimum width=1cm,
                  },
                 ]

        \begin{pgfonlayer}{background}    
            \node[fit=(l11),intree] (layer1) {};
            \node[fit=(l21) (l22),intree] (layer2) {};
            \node[fit=(l31) (l34),intree] (layer3) {};
            \node[fit=(v1) (vm),intree] (layerl) {};
        \end{pgfonlayer}
    \end{scope}
    \node [cloud, draw,
            fill=white,
            cloud puffs=29.7,
            cloud puff arc=120,
            cloud ignores aspect,
            minimum height=3em,
            fit=(r1p)
                (r2p),
            ] {};
   \end{tikzpicture}
\caption{ }  \label{fig:bigex}
 \end{figure}
 With $r_\sigma$ and $h_\sigma$ we denote the corresponding subtrees rooted at the indicated nodes. Note that the $\alpha$-equivalence classes in $\varrho L_\alpha$ of each $r_\sigma$ and of any of its subtrees are singletons since $r_\sigma$ does not contain any \l-abstraction. For this reason we shall henceforth abuse notation and denote those equivalence classes by their representatives. Observe further that the group action on $\varrho L_\alpha$ satisfies $\tau \cdot r_\sigma = r_{\tau\sigma}$ for any $\tau, \sigma \in \perms(\V)$. This implies that all $r_\sigma$ and $h_\sigma$, respectively, lie in the same orbit. Similarly, one can see that all nodes on the same level in the right-hand maximal subtrees of every $r_\sigma$ (indicated by the dotted rectangles) lie in the same orbit. 

Now consider $r_\id$ and the corresponding orbit-finite subcoalgebra $X$ of $\varrho L_\alpha$ from the proof of Theorem~\ref{thm:rat}. The elements of $X$ are all subtrees of $r_\id$ with the group action inherited from $\varrho L_\alpha$. By the above reasoning we see that $X$ has precisely $\ell + 2$ orbits. Hence, the number of orbits of $X$ is logarithmic in $m$. But the number of subtrees of $r_\id$ grows faster than $m!$. To see this, notice first that the permutations $(1\ 2)$ and $(1\ 2\ \cdots\ m)$ generate the group of all permutations of $m$ elements. Thus, we see that $r_\id$ has all $r_\sigma$ as subtrees where $\sigma$ is any permutation that fixes $\V \setminus V$. For different $\sigma$ and $\tau$ we have that $r_\sigma$, $h_\sigma$, $r_\tau$ and $h_\tau$ are pairwise non-isomorphic. Thus, we see that $r_\id$ has at least $2\cdot m!$ subtrees. In addition, consider the $i$-th level (from 1 at the bottom to $\ell$ at the top) on the right in Figure~\ref{fig:bigex}. Each node on that level covers $k = 2^{i-1}$ variables. Then for each permutation of those $k$ variables there exists a subtree of the right-hand successor of some subtree $r_\sigma$ of $r_\id$ which is a complete binary tree of height $i$ with the front
given by the permutation. Thus, the total number of subtrees of $r_\id$ is precisely
\[
2\cdot m! + \sum_{i = 1}^{\ell} \left( {m \choose 2^{i-1}}\cdot (2^{i-1})!\right)
=
2\cdot m! + \sum_{i=1}^\ell \frac{m!}{(m-2^{i-1})!}.
\]
%
\end{example}

\section{Application: Corecursive Definitions on Rational $\lambda$-trees}
\label{sec:app}

Our result in Theorem~\ref{thm:rat} that rational \l-trees modulo
$\alpha$-equivalence form the final locally orbit-finite
$L_\alpha$-coalgebra yields a corecursion principle. In this section
we shall demonstrate this principle by considering two easy
applications. First we show that substitution as defined corecursively
for all \l-trees in~\cite{nomcoalgdata} restricts to rational
\l-trees. Secondly, we discuss the corecursive definition of the
computation of the Böhm tree of a given rational \l-tree.

\subsection{Substitution on Rational $\lambda$-trees}\label{sec:sub}
When performing operations known from (infinitary) \l-calculus on rational
\l-trees, it is not clear whether the resulting \l-tree still is rational in
general. One such operation is the substitution function
$\subs: \nu L_\alpha × \V ×\nu L_\alpha \longrightarrow \nu L_\alpha$,
which for a given triple $(t,v,s)$ replaces each occurence of the
variable $v$ in $t$ by $s$. Kurz et al.~\cite{nomcoalgdata} show how
to define \subs by coinduction: to do this they define an $L_\alpha$-coalgebra whoose unique
homomorphism into the final coalgebra yields \subs. It is possible to adapt this to
the orbit-finite case as follows.

For arbitrary coalgebras $A\xrightarrow{a}L_\alpha A$ and
$B\xrightarrow{b}L_\alpha B$ with orbit-finite carriers, one defines
$\subs_{A,B}: A×\V×B \to \varrho L_\alpha$. This map $\subs_{A,B}$ describes
the substitution of a variable within an element of the coalgebra
$A$ by some element of the coalgebra $B$. In the following
\[
B \xrightarrow{\mathsf{inl}} B + A \times \V \times B
\xleftarrow{\mathsf{inr}} A \times \V \times B
\qquad\text{and}\qquad
\V \xrightarrow{\mathsf{in}_1} L_\alpha X \xleftarrow{\mathsf{in}_2}
[\V] X
\]
denote coproduct injections. Now define
\(
    B+ A×\V×B \xrightarrow{[g,h]} L_\alpha(B+ A×\V×B)
\)
with $g= L_\alpha (\mathsf{inl})\circ b$ and
\(
    h = [h_\text{Var}, h_\text{Abs}, h_\text{App}] \circ (a×\id×\id)
\) using distributivity and
\begin{itemize}
\item \(
    h_\text{Var}: \V × \V×B \to L_\alpha(B+A×\V×B)
\),
\(
    h_\text{Var}(v,w,x) =
    \begin{cases}
        \mathsf{in}_1 v & \text{if }v\neq w \\
        L_\alpha(\mathsf{inl})(\underbrace{b(x)}_{\mathclap{\in L_\alpha(B)}}) &
            \text{if }v=w
    \end{cases}
\)

\item $h_\text{Abs}: ([\V]A)×\V×B \to L_\alpha (B+A×\V×B)$,
     $h_\text{Abs} := \mathsf{in}_2 \circ \tau_{A,\V×B}$, using the strength
     $\tau$ of the functor $[\V]$ with $\tau_{A,\V×B}: ([\V]A)×\V×B \to
     [\V](A×\V×B)$.

\item $h_\text{App}: A×A×\V×B \to L_\alpha (B+A×\V×B)$ is the composition of the
obvious steps
\(
    A×A×\V×B \to (A×\V×B)×(A×\V×B) \to L_\alpha (B+ A×\V×B).
\)

\end{itemize}
$B + A× \V × B$ is orbit-finite, so we get a unique $L_\alpha$-coalgebra
homomorphism $[g',h']$ into $\varrho L_\alpha$, making the following diagram commute:
\[
    \begin{tikzcd}[compactcd]
        B + A× \V × B
            \arrow{d}[left]{[g',h']}
            \arrow{r}{[g,h]}
            & 
        L_\alpha(B + A× \V × B)
            \arrow{d}[right]{L_\alpha([g',h'])}
            \\
        \varrho L_\alpha
            \arrow{r}{f_\alpha}
            &
            L_\alpha(\varrho L_\alpha)
    \end{tikzcd}
\]
Now we define $\subs_{A,B} := h'$. It remains to extend this to
the desired domain $\varrho L_\alpha × \V × \varrho L_\alpha$. Let $I$
be the set of all orbit-finite subcoalgebras $(X,a)$
of $(\nu L_\alpha, t_\alpha)$. Then we know
from (the discussion following) Proposition~\ref{prop:image} that
\[
    \varrho L_\alpha = \bigcup\limits_{(X,a) \text{ in  } I} X.
\]
Hence, we can define substitution $\subs_\text{rat}: \varrho L_\alpha × \V × \varrho L_\alpha \to \varrho L_\alpha$ on rational \l-trees by
\[
    \subs_\text{rat} (x,v,y) = \subs_{A,B}(x,v,y), \text{for some $A,
      B \in I$ with $x\in A$ and $y\in B$}.
\]
It remains to prove that the result does not depend on the choice of $A$ and
$B$. But if we have any other $(A', a')$ in $I$ with $x \in A'$ then
since both $A$ and $A'$ are subcoalgebras we have $a(x) = t_\alpha(x)
= a'(x)$. Similarly for $B$ and $y \in B$. 

Thus, since the function $h$ is defined by pattern matching, i.e.~on the alternatives indicated by $a$, it
behaves independently from the choice of $A$ and $B$ as desired. 

To summarize, we can say that one can define operations on $\varrho L_\alpha$ if these $n$-ary operations can be defined restricted to $n$ orbit-finite subcoalgebras of $\nu L_\alpha$ as we have just seen for the operation of substitution.

\subsection{Normalization of Rational $\lambda$-trees}

In the $\l$-calculus different kinds of normal forms play an important role. One of them is the \emph{head normal form} (\emph{hnf}, for short). A $\l$-term is in hnf if it is of the form
$\l x_1\ldots\l x_n.yN_1\ldots N_m$, where $y$ is a variable and the $N_i$ are arbitrary terms. If one recursively requires the $N_i$ to be in hnf as well, one gets the definition of \emph{Böhm trees}. Adding an additional constant symbol $\bot$ to the syntax of the $\l$-calculus allows the following corecursive definition of the Böhm tree $\op{BT}(M)$ of a \l-term $M$ (see~\cite{nomcoalgdata}): 
\begin{align}
    \op{BT}(M) = \begin{cases}
        \l x_1\ldots\l x_n.y \op{BT}(N_1)\ldots \op{BT}(N_m)
        &
        \text{if }
        M\twoheadrightarrow_\beta
        \l x_1\ldots\l x_n.y N_1\ldots N_m
        \\
        \bot&\text{otherwise},
    \end{cases}
    \label{bt}
\end{align}
where $\epito_\beta$ denoted the reflexive, transitive closure of $\beta$-reduction $\to_\beta$. 
So a Böhm tree of a term $M$ is the normal form of $M$ in the infinitary $\l$-calculus, or $\bot$ if there is no normal form \cite{bar84,KKSdV97}.

Kurz et al.~\cite{nomcoalgdata} obtained the operation $\op{BT}$ by using the final coalgebra  $\Lambda_\alpha^\infty$ of the following endofunctor $\Lb$ on $\Nom$ expressing the syntax of the $\l$-calculus extended by $\bot$:
\[
    \Lb X = \V + \{\bot\} + [\V]X + X×X.
\]
The nominal set $\Lambda_\alpha^\infty$ consists of all $\l$-trees over $\bot$ modulo $\alpha$-equivalence, i.e.~\l-trees where some leaves are labelled by $\bot$ in lieu of a variable from $\V$. So $\op{BT}: \Lambda_\alpha^\infty \to \Lambda_\alpha^\infty$ is defined as the unique coalgebra homomorphism from a coalgebra $b: \Lambda_\alpha^\infty \to \Lb(\Lambda^\infty_\alpha)$, where $b$ is defined by 
\[
b(M) = \begin{cases}
  t(N) & \text{if $M \epito_\beta N$ and $N$ is in hnf}\\
  \bot & \text{else},
\end{cases}
\]
into the final coalgebra $\nu \Lb$. 

The rational fixpoint $\varrho \Lb$ consists of the rational \l-trees over $\bot$ modulo $\alpha$-equivalence; in fact, it is easy to extend to the proof of Theorem~\ref{thm:rat} to the functor $\Lb$.

However, $\op{BT}$ does not restrict to $\varrho \Lb$. Consider the $\l$-term
\begin{align}
    u := Y_{(\l g.\l x.x\,(g(x\,y)))},
    \text{ with } Y_f := (\l z.f (z\ z)) (\l z.f (z\ z)).
    \label{def:tY}
\end{align}
which is finite and therefore rational. In other words, $u$ represents an element of $\varrho \Lb$. Let us look at its Böhm tree, by considering the $\beta$-reduction sequence of $u$ using $Y_f \rightarrow_\beta f(Y_f f)$:
\begin{align*}
  u \,x
    &= Y_{(\l g.\l x.x\,(g(x\,y)))}\,x
    \rightarrow_\beta (\l g.\l x.x\, (g(x\,y)))\,
                      \overbrace{Y_{(\l g.\l x.x\,(g(x\,y)))}}^{\smash{u}}\,x
    \twoheadrightarrow_\beta x\,(u\,(x\,y))
    \label{trecursion}
\end{align*}
Applying $u x\twoheadrightarrow_\beta x\,(u\,(x\,y))$ multiple times yields the following sequence:
\[
\begin{tikzpicture}[lambdatree,mathnodes]
    \node (z) {@}
    child { node[subtree] {u} }
    child { node {x} }
    ;
\end{tikzpicture}
\ \twoheadrightarrow_\beta\ %
\begin{tikzpicture}[lambdatree,mathnodes]
    \node (z) {@}
    child { node {x} }
    child {
        node {@}
            child { node[subtree] {u} }
            child { node {@}
                child { node {x} }
                child { node {y} }
            }
        }
    ;
\end{tikzpicture}
\twoheadrightarrow_\beta%
\begin{tikzpicture}[lambdatree,mathnodes]
    \node (z) {@}
    child { node {x} }
    child { node {@}
            child { node {@}
                child { node {x} }
                child { node {y} }
            }
            child {
                node {@}
                    child { node[subtree] {u} }
                    child { node {@}
                        child { node {@}
                            child { node {x} }
                            child { node {y} }
                        }
                        child { node {y} }
                    }
                }
        }
    ;
\end{tikzpicture}
\twoheadrightarrow_\beta
\begin{tikzpicture}[lambdatree,mathnodes]
    \node (z) {@}
    child { node {x} }
    child { node {@}
            child { node {@}
                child { node {x} }
                child { node {y} }
            }
            child { node {@}
                child { node {@}
                    child { node {@}
                        child { node {x} }
                        child { node {y} }
                    }
                    child { node {y} }
                }
                child[noedge] { node[anchor=west,yshift=2mm] {\ddots} }
            }
        }
    ;
\end{tikzpicture}
\twoheadrightarrow_\beta\cdots
\]
The resulting infinite \l-tree is clearly not rational; in fact, consider the subtrees defined by the left-hand cildren of every node on the right-most path. Then the subtree rooted at the left-hand sucessor of the $n$-th node on that path has the list $xy^n$ as its front of leaf labels. And since this tree does not contain any \l-operators its $\alpha$-equivalence class is a singleton, whence $\op{BT}(u) \not \in \varrho \Lb$. 

\takeout{ 
This infinite tree, let us call it $N$, is not rational, and thus not in $\varrho \Lb$ but in $\nu \Lb$. As $N$ has no abstractions at all, it is a Böhm tree, and we can conclude:
\[
    \varrho \Lb \ni u\ x \overset{BT}{\longmapsto}
    N
    = x\, \big((x\,y)\ \big((x\,y\,y) \big(\cdots \Big)
    \not\in \varrho \Lb.
\]
So $\op{BT}$ does not restrict to $\varrho \Lb$, or in other words the corresponding coalgebra  is not lfp.
} 
Of course, there are also \l-terms, whose Böhm tree is infinitely large but stays rational, for example:
\[
    s := Y_{(\l g.\l x.\l y.x\,g\,y)))},
    \text{ with }
    s
    \rightarrow_\beta
    (\l g.\l x. \l y.x\,g\,y)))\,s
    \rightarrow_\beta
    \l x.\l y.x\,s\,y.
    \twoheadrightarrow_\beta\cdots \twoheadrightarrow_\beta
    \hspace{-8mm}
\begin{tikzpicture}[lambdatree,mathnodes]
    \node (z) {\l x}
    child { node[yshift=-1mm] {\l y}
        child { node[yshift=-1mm]  {@}
            child { node (chld) {@}
                child { node {x} }
                child[noedge] { node { } }
            }
            child { node {y} }
        }
    };
    \begin{scope}
    \path[use as bounding box]
        ($ (z) + (-1cm,5mm) $) rectangle (z);
    \path[draw] (chld) .. controls ($ (chld) - (-2cm,2cm) $)
                             and ($ (z) + (1cm,2cm) $) .. (z);
    \end{scope}
\end{tikzpicture}
\]
The rational \l-tree on the right above, call it $r$, is the Böhm tree for $s$, i.e. $\op{BT}(s)=r\in \varrho \Lb$. In other words the subcoalgebra $S$ of $(\Lambda^\infty_\alpha, b)$ above generated by $[s]_\alpha$ is orbit-finite, hence the restriction of $\op{BT}$ to $S$ factorizes through $\varrho \Lb$. 

Can one characterize the largest subcoalgebra of $(\Lambda^\infty_\alpha, b)$ whose image under $\op{BT}$ lies in $\varrho \Lb$? We leave this question for further work.

\section{Conclusion and Future Work}
\label{sec:con}

We have contributed to the abstract algebraic study of
variable binding using nominal sets. In particular, we have extended a
recent coalgebraic approach to infinitary \l-calculus due to Kurz et
al. Whereas they proved in~\cite{nomcoalgdata} that \l-trees with
finitely many variables modulo $\alpha$-equivalence form the final
coalgebra for the functor $L_\alpha$ on $\Nom$ we have given a
characterization of the rational fixpoint of that functor. It contains
precisely the rational \l-trees modulo $\alpha$-equivalence. 

This characterization entails a corecursion principle for rational
\l-trees because the rational fixpoint is the final locally orbit-finite
coalgebra for $L_\alpha$. In this sense we have achieved finitary
corecursion for the infinitary \l-calculus. We have demonstrated the new principle and its
limitations with two applications: a corecursive definition of substitution and of a normalform computation. 

Our work is only a first step in the study of the coalgebraic approach
to finitary coinduction for infinitary terms with variable binding
operators. First, it should be clear that our results generalize from
\l-terms to the rational fixpoint for endofunctors on $\Nom$
associated to a binding signature. Other points for future work are:
(1)~the extension of the coalgebraic approach to rational and infinitary
\l-terms using nominal sets to treat the solutions of higher-order
recursion schemes as was done in the setting of presheaves on finite
sets in~\cite{highrecursion}, and (2)~the study of specification formats
that extend our simple corecursion principle that follows from
finality; more precisely, Bonsague et al.~\cite{bmr12,mbmr13} have proposed
bipointed specifications as an abstract format (by restricting Turi's and
Plotkin's abstract GSOS rules~\cite{tp97}) to specify algebraic
operations on the rational fixpoint of an endofunctor. It should be
interesting to work out a concrete rule format corresponding to bipointed
specifications for rational \l-terms and rational terms for arbitrary
binding signatures. Last, but not least, the similarity of the results
in~\cite{highrecursion} on the one hand and those
in~\cite{nomcoalgdata} and here on the other hand is so striking that there should be a
formal connection; however, to our knowledge this has not been worked out in the
literature yet.

\takeout{
Some points for future work:
\begin{itemize}
\item Generalize to arbitrary bindung signatures
\item Work on solutions of recursion schemes in nominal sets
\item Work out GSOS formats (in particular, bipointed specifications) for rational lambda-trees
\item Formal connection to recursion scheme paper by Ad\'amek et al.
\end{itemize}} 

%
%
\bibliographystyle{plain}
\bibliography{refs}

%
%

\iffull
\clearpage
\appendix

\section{Proof of Lemma~\ref{orbitsamesupp}}
\begin{proof}
    Assuming $\pi \cdot x_1 = x_2$, the equivariance of \supp provides the equation
    \[
        \supp(x_1) = \supp(\pi\cdot x_2) = \pi\cdot \supp(x_2)
        = \{ \pi (v) \mid v \in \supp(x_2) \}.
    \]
    Since $\pi$ is a bijective map this establishes the desired bijection between $\supp(x_1)$ and $\supp(x_2)$.
\end{proof}

\section{Proof of Lemma~\ref{sameorbitfaculty}}

\begin{proof}
    Consider the set
    \[
        N = \{ \pi \cdot x \mid \pi \in \perms(\V)\text{ and } \supp(\pi\cdot x) = \supp(x)   \}
    \] 
    All $\pi$ with $\pi\cdot x \in N$ have to fulfill the property $\supp(x) = \supp(\pi \cdot x) = \pi \cdot \supp(x)$. This means, $\pi$ restricts to $\supp(x)$, i.e.~$\pi(v) \in \supp(x)$ iff $v \in \supp(x)$. Hence, each such $\pi$ can be factorized into two permutations $f_\pi,g_\pi\in \perms(\V)$ where
    \[
        f_\pi(v) = \begin{cases}
            \pi(v) &\text{if }v\in \supp(x) \\
            v      &\text{if }v\not\in \supp(x) \\
        \end{cases}
        ,\quad
        g_\pi(v) = \begin{cases}
            v &\text{if }v\in \supp(x) \\
            \pi(v)      &\text{if }v\not\in \supp(x) \\
        \end{cases}
    \]
    and with $f_\pi\circ g_\pi = g_\pi \circ f_\pi = \pi$. Since $\supp(x)$ supports $x$ we have $g_\pi \cdot x = x$, and this implies
    \begin{align*}
        N &= \{ f_\pi\cdot g_\pi \cdot x
                \mid \pi \in \perms(\V) \text{ and } \pi\cdot \supp(x) = \supp(x)
            \}
        \\
          &= \{ \mathrlap{f_\pi \cdot x}\phantom{f_\pi\cdot g_\pi \cdot x }
                \mid \pi \in \perms(\V) \text{ and } \pi\cdot \supp(x) = \supp(x)
            \}.
    \end{align*}
    The set of permutations $\{f_\pi \mid \pi\cdot x \in N\}$ has cardinality $|\supp(x)|!$, hence so is $N$.
\end{proof}

\section{Proof of Lemma~\ref{contained}}
\begin{proof}
  First there are only finitely many elements whose support is exactly $W$. If there is no element in $\mathcal{O}$ with support $W$, then we are done. And if $x$ in $\mathcal{O}$ satifies $\supp(x) = W$, there are only finitely many elements in $\mathcal{O}$ with the same support by \autoref{sameorbitfaculty}.
  
  The statement of the lemma now follows since $W$ has only finitely many subsets and for each $W' \subseteq W$ there are only finitely many elements with support $W'$.
\end{proof}

\takeout{ 
\section{Proof of Proposition~\ref{xfxsamesupport}}

\begin{proof}
    Combine \autoref{orbitsamesupp} with the inclusion $\supp(f(x)) \subseteq \supp(x)$ and use that the
    support is finite:
    \[
        \left.
        \begin{aligned}
        |\supp(f(x))| &= |\supp(x)| \\
        \supp(f(x)) \phantom{|}&\subseteq \phantom{|}\supp(x)
        \end{aligned}\right\}
        \Longrightarrow
        \supp(x) = \supp(f(x))
        \qedhere
    \]
\end{proof}
} 

\section{Proof of Proposition~\ref{prop:ratL}}

\begin{proof}
Take a tree $t \in \varrho L$ and an orbit-finite subcoalgebra $(S,s)$ of $\varrho L$ with $t \in S$. Every subtree $t'$ of $t$ has support 
\[
    \supp(t') \subseteq \supp(t).
\]
But there are only finitely many $t' \in S$ with that property; this follows from \autoref{contained} using $W = \supp(t)$ and that $S$ is orbit-finite. Hence, $t$ has only finitely many subtrees, i.e.~$t$ is rational. So $\varrho L$ contains only rational \l-trees. Conversely, for each rational \l-tree $t$ we can construct an orbit-finite $L$-coalgebra by taking the subtrees of $t$ as the carrier set of that coalgebra and with the coalgebra structure given by decomposing trees at the root.
\end{proof}

\section{Proof of Theorem~\ref{thm:onlyrational}}

    Let $m$ be the maximal number of free variables in any element of $X$, i.e.
    \[
        m := \max_{x\in X}\big|\supp(x)\big|.
    \]
    The maximum exists by \autoref{orbitsamesupp} since $X$ is orbit-finite.

    Let $W \subseteq \V$ be some set of $m+1$ variables containing $\supp(\op{root})$. Hence 
    \begin{equation}
      \text{for all $x\in X$ there exists a $w\in W$ with $w \fr  x$.}
      \label{eq:alwayssomefresh}
    \end{equation}
      
    In the following, we will construct a rational $\l$-tree in the $\alpha$-equivalence class of $a^\dagger(\op{root})$.
    First, define an $L$-coalgebra $C \xrightarrow{c} LC = \V + \V\times C + C\times C$ in \Set:
    \[
        C := \{x \in X\mid \supp(x) \subseteq W \},
        \quad
        c(x) = \begin{cases}
            w & \text{if }a(x) = w \in W \subseteq \V \\
            (\ell, r) &\text{if }a(x) = (\ell,r)\in X\times X\\
            (w,y) &\text{if }a(x) = \abstr{v} y'\text{ and }y = (v\ w)y'\\
            &\text{for some }w\in W\setminus\supp(x)
        \end{cases}
    \]
    We verify that $c$ is well-defined, i.e., its image lies in $LC$:
    \begin{itemize}
    \item For the case $a(x)$ in $\V$, $a(x)$ is also in $W$ because $x \in C$
          and thus
          \[
              \supp(a(x)) \subseteq \supp(x) \subseteq W.
          \]
    \item For the case $a(x) = (\ell,r)$ in $X\times X$,
    \[
        \supp(\ell)\cup\supp(r) \subseteq \supp(x) \subseteq W.
    \]
    So $\ell, r \in C$.

    \item For $a(x) = \abstr{v}y' \in [\V]X$, we know by
    \eqref{eq:alwayssomefresh}, that there is such a fresh $w \in W$. So
    $\abstr{w}y = \abstr{v}y'$. In particular, $y\in C$, because
    \[
        \supp(y) \subseteq \{w\} \cup \supp(a(x))
                 \subseteq \{w\} \cup \supp(x)
                 \subseteq \{w\} \cup W \subseteq W.
    \]
    \end{itemize}
    The last item gives that $\abstr{w}y = \abstr{v}y'$, hence the following diagram commutes in \Set:
    \begin{equation}
      \begin{tikzcd}[compactcd]
        C \rar{c} \dar[hook][swap]{i}
        &
        LC 
        \rar[hook]{Li}
        &
        LX
        \dar{q_X}
        \\
        X
        \ar{rr}[swap]{a}
        &&
        L_\alpha X
      \end{tikzcd}
      \takeout{ 
      \begin{tikzcd}[compactcd]
        C \arrow[hook]{r}{} \arrow{d}[left]{c}& X \arrow{r}{a} & L_\alpha X
        \mathrlap{\ =\V + [\V]X + X×X}
        \\
        LC
        \arrow[hook]{rr}{}
        &&
        LX
        \mathrlap{\ = \V +\V×X + X×X}
        \arrow{u}[right]{\id_\V + \operatorname{bind} + (\id_X×\id_X)}
    \end{tikzcd}}
        \label{eq:coalgplusbind} 
    \end{equation}
    \takeout{Here, bind denotes the equivariant $\V×X \to [\V]X$, $(v,x)\mapsto
    \abstr{v}x$. (Note also that we use that coproducts and finite products
    in \Nom are build the same way as in \Set.) The commutativity of the diagram
    tells us that the $L$-coalgebra $(C,c)$ in \Set plus binding behaves the
    same as the $L_\alpha$-coalgebra $(X,a)$ in \Nom.}

    Observe that $C$ is finite, because $X$ is orbit-finite and within
    an orbit there are only finitely many elements with a support
    contained in $W$ by \autoref{contained}.  Let $c^\dagger$
    denote the unique $L$-coalgebra morphism into the final
    $L$-coalgebra in \Set. Since $C$ is finite, we know that $c^\dagger:
    C \to \nu L$ factors through the rational fixpoint, i.e.~for every
    $x \in C$, $c^\dagger(x)$ is a rational $\lambda$-tree. In particular, $c^\dagger(\op{root})$ is a rational $\lambda$-tree. Using~\eqref{eq:coalgplusbind} and the \emph{final chains} of the set functor $L$ and the functor $L_\alpha$ on $\Nom$ we shall prove below that 
    \begin{equation}\label{eq:ca}
      \eqa{c^\dagger(x)} = a^\dagger(x)\qquad\text{for all $x\in C\subseteq X$},
    \end{equation}
    where $\eqa{-}$ denotes $\alpha$-equivalence classes.
    For $x = \op{root}$ this gives that $c^\dagger(\op{root})$ is the desired
    rational \l-tree in $a^\dagger(\op{root})$ and thus the $\alpha$-equivalence
    class $a^\dagger(\op{root})$ is rational.

    We will now prove Equation~\eqref{eq:ca}. Before we can proceed to the proof we need to recall some technical background (cf.~\cite{nomcoalgdata}). First recall that every endofunctor $F$ on a category $\C$ with a final object $1$ induces its final chain
\[
1 \leftarrow F1 \leftarrow FF1 \leftarrow \cdots 
\]
Whenever the limit of this chain exists in $\C$ and $F$ preserves this limit, then it carries the final $F$-coalgebra $\nu F$. Moreover, every $F$-coalgebra $c: C \to FC$ induces its \emph{canonical cone} $c^n: C \to F^n 1$, $n < \omega$, on the final chain defined inductively as: $c^0: C \to 1$ is the unique map, and given $c^n$ we define 
\[
c^{n+1} = C \xrightarrow{c} FC \xrightarrow{Fc^n} FF^n 1 = F^{n+1} 1
\qquad
\text{for all $n > 0$}. 
\]
The unique $F$-coalgebra homomorphism $c^\dagger: C \to \nu F$ is then induced by the universal property of the limit -- $c^\dagger$ is the unique morphism such that 
\begin{equation}
  \label{eq:cn}
  c^n = (C \xrightarrow{c^\dagger} \nu F \xrightarrow{p_n} F^n 1) \qquad \text{for every $n <\omega$,}
\end{equation}
where $p_n: \nu F \to F^n 1$ denote the limit projections. 
\takeout{ 
Note also that coalgebra homomorphisms ``preserve'' canonical cocones in the following sense:
\begin{lemma}
  Let $h: (A,a) \to (B,b)$ be an $F$-coalgebra homomorphism. Then we have
  \[
  a^n = (A \xrightarrow{h} B \xrightarrow{b^n} F^n 1)
  \qquad
  \text{for all $n < \omega$}. 
  \]
\end{lemma}
\begin{proof}
  By induction on $n$: the base case follows since $1$ is a final object, and for the induction step one computes:
\[
a^{n+1} = Fa^n \cdot a = Fb^n \cdot Fh \cdot a = Fb^n \cdot b \cdot h = b^{n+1} \cdot h. \qedhere
\]
\end{proof}}

Now let us turn to our concrete functors of interest. The terminal coalgebra for the functor $L: \Set \to \Set$ consists of all finite and infinite $\lambda$-trees. Since $L$ preserves all limits, $\nu L$ is the limit of its final chain $L^n 1$ and we denote the limit projections by 
\[
p_n: \nu L \to L^n 1
\] 
and the canonical cone induced by $c: C \to LC$ by 
\[
c^n: C \to L^n1.
\] 

We also consider the final chain of the functor $L_\alpha$ on $\Nom$, and we take its limit in $\lim L_\alpha^n 1$ in $\Set$ with the projections 
\[
p_{\alpha,n}: \lim L_\alpha^n 1 \to L_\alpha^n 1, 
\qquad
\text{for $n < \omega$}. 
\]
Notice that this limit does \emph{not} carry the final $L_\alpha$-coalgebra. But $\nu L_\alpha$ is a subset of the limit via the unique map $\iota_\alpha: \nu L_\alpha \to \lim L_\alpha^n 1$ induced by the canonical cone $(t_\alpha^n)_{n < \omega}$ induced by the structure $t_\alpha: \nu L_\alpha \to L_\alpha(\nu L_\alpha)$ of the final coalgebra, i.e., $\iota_\alpha$ is the unique map such that 
\begin{equation}
  \label{eq:iota}
  t_\alpha^n = (\nu L_\alpha \xrightarrow{\iota_\alpha} \lim L_\alpha^n 1 \xrightarrow{p_{\alpha,n}} L_\alpha^n 1)
  \qquad
  \text{for all $n < \omega$.}
\end{equation}
We will also use the canonical cone 
\[
a^n: X \to L_\alpha^n 1, \qquad n < \omega,
\] 
of our given $L_\alpha$-coalgebra $(X,a)$. 

As a final ingredient we need to relate the final chains of $L$ and $L_\alpha$. To this end observe that $L$ can be construed as an endofunctor on $\Nom$ (more precisely, $L$ lifts to $\Nom$), and the we have a quotient natural transformation $q: L \to L_\alpha$ with components
\[
q_X = (LX = \V + \V \times X + X \times X \xrightarrow{\id + \theta_X + \id \times \id} \V + [\V] X + X = L_\alpha X,
\]
where $\theta_X: \V \times X \to [\V] X$ is the natural transformation given by the canonical quotient maps (cf.~Definition~\ref{def:abstr}).
Using $q$ one can define a natural transformation $\eqa{-}^n: L^n 1 \to L_\alpha^n 1$, $n < \omega$, between the respective final chains by induction: $\eqa{-}^0$ is the identity on $1$ and 
\[
\eqa{-}^{n+1} =  L^{n+1} 1 = LL^n 1 
\xrightarrow{L\eqa{-}^n} 
LL_\alpha^n 1 
\xrightarrow{q_{L_\alpha^n 1}}
L_\alpha L_\alpha^n 1 = L_\alpha^{n+1} 1.
\]
Now we regard the two final chains and the natural transformation given by the $\eqa{-}^n$ in $\Set$ and use the universal property of the limit of the $L_\alpha^n 1$ to obtain a unique map $\eqa{-}: \nu L \to \lim L_\alpha^n 1$ such that the following squares commute for all $n < \omega$:
\begin{equation}\label{eq:eqa}
  \begin{tikzcd}[compactcd]
    L^n 1 \ar{d}[swap]{\eqa{-}^n} & 
    \nu L \ar{l}[swap]{p_n}  \ar{d}{\eqa{-}}
    \\
    L_\alpha^n 1 & 
    \lim L_\alpha^n 1
    \ar{l}{p_{\alpha,n}}
  \end{tikzcd}
\end{equation}
We denote by $i: C \hra X$ the inclusion map and consider the following diagram:
\begin{equation}
  \label{diag:zz}
\begin{tikzcd}[compactcd]
  \nu L
  \ar{d}[swap]{\eqa{-}}
  \ar[hookleftarrow]{r}
  &
  \Tffv
  \ar[->>]{d}
  &
  C 
  \ar[shiftarr={yshift=3.5ex}]{ll}[swap]{c^\dagger} 
  \ar[hook]{d}{i} 
  \ar[dashed]{l}
  \\
  \lim L_\alpha^n 1
  \ar[hookleftarrow]{r}[swap]{\iota_\alpha}
  &
  \nu L_\alpha
  &
  X
  \ar{l}{a^\dagger}
\end{tikzcd}
\end{equation}
The set $\Tffv$ consists of all \l-trees with finitely many free variables, and this makes the left-hand square above the pullback of $\eqa{-}$ along $\iota_\alpha$ as proved in~\cite[Proposition~5.33]{nomcoalgdata}. There it was also shown that the map opposite $\eqa{-}$ is surjective (see~\cite[Theorem~5.34]{nomcoalgdata}), and it follows that $\nu L_\alpha$ consists of all equivalence classes modulo $\alpha$-equivalence of \l-trees in $\Tffv$. 

Now note that once we prove that the outside of Diagram~\eqref{diag:zz} commutes we obtain the dashed map making the diagram commutative. This finishes our proof since the right-hand square of~\eqref{diag:zz} is precisely our desired equation~\eqref{eq:ca}. We now establish the commutativity of the outside of the diagram by proving inductively that 
\[
p_{\alpha,n} \o \iota_\alpha \o a^\dagger \o i = p_{\alpha,n} \o \eqa{-} \cdot c^\dagger,
\qquad
\text{for all limit projections $p_{\alpha, n}$, $n < \omega$}.
\]
The base case is clear and for the induction step we compute:
\enlargethispage{10pt}
\[
\begin{array}{rcl@{\qquad}p{8cm}}
  p_{\alpha, n+1} \o \iota_\alpha \o a^\dagger \o i 
  & = & t_\alpha^{n+1} \o a^\dagger \o i & by~\eqref{eq:iota} \\
  & = & L_\alpha t_\alpha^n \o t_\alpha \o a^\dagger \o i & definition of $t_\alpha^{n+1}$ \\
  & = & L_\alpha t_\alpha^n \o L_\alpha a^\dagger \o a \o i & $a^\dagger$ coalgebra homomorphism \\
  & = & L_\alpha t_\alpha^n \o L_\alpha a^\dagger \o q_X \o Li \o c & see~\eqref{eq:coalgplusbind} \\
  & = & q_{L_\alpha^n 1} \o L t_\alpha^n \o L a^\dagger \o Li \o c & $q$ natural and $t_\alpha^n$, $a^\dagger$ equivariant \\
  & = & q_{L_\alpha^n 1} \o L p_{\alpha,n} \o L \iota_\alpha \o L a^\dagger \o Li \o c & by~\eqref{eq:iota} \\
  & = & q_{L_\alpha^n 1} \o L(p_{\alpha,n} \o \iota_\alpha \o a^\dagger \o i) \o c & functoriality of $L$ \\
  & = & q_{L_\alpha^n 1} \o L(p_{\alpha,n} \o \eqa{-} \o c^\dagger) \o c & induction hypothesis \\
  & = & q_{L_\alpha^n 1} \o Lp_{\alpha,n} \o L\eqa{-} \o Lc^\dagger \o c & functoriality of $L$ \\
  & = & q_{L_\alpha^n 1} \o L\eqa{-}^n \o Lp_n \o Lc^\dagger \o c & by~\eqref{eq:eqa} \\
  & = & q_{L_\alpha^n 1} \o L\eqa{-}^n \o Lc^n \o c &  by~\eqref{eq:cn} \\
  & = & \eqa{-}^{n+1} \o Lc^n \o c & definition of $\eqa{-}^{n+1}$ \\
  & = & \eqa{-}^{n+1} \o c^{n+1} & definition of $c^{n+1}$ \\
  & = & \eqa{-}^{n+1} \o p_{n+1} \o c^\dagger & by~\eqref{eq:cn} \\
  & = & p_{\alpha,n+1} \o \eqa{-} \o c^\dagger & by~\eqref{eq:eqa}
\end{array}
\]
This completes the proof of the desired equation~\eqref{eq:ca}.\qed
\fi

\end{document}